\documentclass{amsart}
\usepackage{amssymb,latexsym}
\usepackage{amscd,amsthm}
\usepackage{tikz}

\usepackage[all]{xy}

\newtheorem{theorem}{Theorem}[section]
\newtheorem{lemma}[theorem]{Lemma}
\newtheorem{proposition}[theorem]{Proposition}
\newtheorem{corollary}[theorem]{Corollary}

\theoremstyle{definition}
\newtheorem{definition}[theorem]{Definition}

\newtheorem{remark}[theorem]{Remark}
\newtheorem*{notation}{Notation}

\DeclareMathOperator{\Ext}{Ext}

\DeclareMathOperator{\Hom}{Hom}
\DeclareMathOperator{\Tor}{Tor}

\DeclareMathOperator{\cok}{cok}
\DeclareMathOperator{\im}{Im}

\newcommand{\cat}[1]{\mathcal{#1}}           

\newcommand{\class}[1]{\mathcal{#1}}   

\newcommand{\Z}{\mathbb{Z}}
\newcommand{\Q}{\mathbb{Q/Z}}
\newcommand{\mathcolon}{\colon\,} 

\newcommand{\ch}{\textnormal{Ch}(R)}

\newcommand{\cha}[1]{\textnormal{Ch}(\mathcal{#1})}

\newcommand{\rmod}{R\text{-Mod}}

\newcommand{\tilclass}[1]{\widetilde{\class{#1}}}
\newcommand{\dgclass}[1]{dg\widetilde{\class{#1}}}

\newcommand{\rightperp}[1]{#1^{\perp}}
\newcommand{\leftperp}[1]{{}^\perp #1}

\newcommand{\homcomplex}{\mathit{Hom}}

\begin{document}

\title{Absolutely clean, level, and Gorenstein AC-injective complexes}

\author{Daniel Bravo}
\address{ Instituto de Ciencias F\'isicas y Matem\'aticas\\
Universidad Austral de Chile \\ Valdivia, Chile.}
\email{daniel.bravo@uach.cl}
\thanks{The first author was partially supported by CONICYT/FONDECYT/INICIACION/11130324}

\author{James Gillespie}
\address{Ramapo College of New Jersey \\
         School of Theoretical and Applied Science \\
         505 Ramapo Valley Road \\
         Mahwah, NJ 07430}
\email[Jim Gillespie]{jgillesp@ramapo.edu}
\urladdr{http://phobos.ramapo.edu/~jgillesp/}

\date{\today}

\begin{abstract}
Absolutely clean and level $R$-modules were introduced in~\cite{bravo-gillespie-hovey} and used to show how Gorenstein homological algebra can be extended to an arbitrary ring $R$. This led to the notion of Gorenstein AC-injective and Gorenstein AC-projective $R$-modules. Here we study these concepts in the category of chain complexes of $R$-modules. We define, characterize and deduce properties of absolutely clean, level, Gorenstein AC-injective, and Gorenstein AC-projective chain complexes.  We show that the category $\ch$ of chain complexes has a cofibrantly generated model structure where every object is cofibrant and the fibrant objects are exactly the Gorenstein AC-injective chain complexes.  
\end{abstract}

\maketitle

\section{Introduction}\label{sec-intro}

The purpose of this paper is to begin the study of Gorenstein AC-homological algebra in the category of chain complexes of $R$-modules. By Gorenstein AC-homological algebra we mean the extension of Gorenstein homological algebra to arbitrary rings $R$ that was recently introduced in~\cite{bravo-gillespie-hovey}. In that paper it is shown how Gorenstein homological algebra can be extended to arbitrary  rings by replacing finitely generated modules with modules of type $FP_{\infty}$. In doing so, injective modules are replaced with what we call absolutely clean modules while flat modules are replaced with the level modules. In turn Gorenstein injective modules are replaced with the so-called Gorenstein AC-injective modules and likewise Gorenstein projective modules are replaced with the Gorenstein AC-projective modules. Although the definitions have changed, it is slight since these definitions coincide with the usual definitions for nice rings, and yet allow for a very nice theory of Gorenstein homological algebra to hold in full generality.  

This paper begins by studying the absolutely clean chain complexes in Section~\ref{sec-abs clean}. We characterize these complexes as the exact chain complexes whose cycle modules are each absolutely clean $R$-modules. We then go on to show that absolutely clean complexes satisfy the same nice properties that the absolutely clean $R$-modules were shown to satisfy in~\cite{bravo-gillespie-hovey}.

In Section~\ref{sec-Gorenstein AC-inj} we introduce and characterize the Gorenstein AC-injective chain complexes. According to Theorem~\ref{them-characterization of Gorenstein AC-injective complexes} these turn out to be the chain complexes $X$ for which each $X_n$ is a Gorenstein AC-injective $R$-module and such that any chain map $f : A \xrightarrow{} X$ is null homotopic whenever $A$ is an absolutely clean complex. This is inspired by a result we learned from~\cite{Ding-Chen-complex-models} where a similar characterization was given for Ding injective complexes. Indeed when $R$ is (left) coherent, Gorenstein AC-injective and Ding injective are the same thing; so this generalizes the result from~\cite{Ding-Chen-complex-models}. We go on in Section~\ref{sec-Gorenstein AC-inj} to prove Theorem~\ref{them-Gorenstein AC-injectives complete cotorsion pair}. This theorem shows that the category $\ch$ of chain complexes has an abelian model structure where each complex is cofibrant and the fibrant objects are precisely the Gorenstein AC-injectve complexes. 

In the last Section~\ref{sec-level complexes} we turn to the dual notions of level and Gorenstein AC-projective complexes. We use the tensor product of chain complexes introduced in~\cite{enochs-garcia-rozas} to study these complexes and derive similar results to the ones obtained in Sections~\ref{sec-abs clean} and~\ref{sec-Gorenstein AC-inj}. We also get chain complex versions of expected results from~\cite{bravo-gillespie-hovey}. For example we see in Corollary~\ref{cor-level complex cotorsion pair} that the level complexes are a covering class in the category of chain complexes. In Corollary~\ref{cor-duality} we see that a perfect duality exists between the absolutely clean complexes and level complexes.  
The reader will notice that we unfortunately have not proved the projective analog to Theorem~\ref{them-Gorenstein AC-injectives complete cotorsion pair}. One would expect the methods used in~\cite{bravo-gillespie-hovey} showing completeness of the Gorenstein AC-projective cotorsion pair in $R$-Mod to generalize to complexes. But even for $R$-Mod this was a quite technical problem, so it appears that this will require further attention in the future.

\

Finally, a note on prerequisites and our notational conventions. We have written this paper with the reader that has encountered the paper~\cite{bravo-gillespie-hovey} in mind. Having said this, all that is required of the reader is a good understanding of modules, chain complexes and homological algebra. We occasionally will use standard results from the theory of cotorsion pairs, for example from the book~\cite{enochs-jenda-book}.

Throughout the paper $R$ denotes a general ring with identity. Everything we do can be written in terms of either left or right $R$-modules. We will favor the left, so that by $R$-module will mean a left $R$-module, unless stated otherwise. The category of $R$-modules will be denoted $\rmod$ and the category of chain complexes of $R$-modules will be denoted $\ch$.

Our convention is that the differentials of our chain complexes \emph{lower} degree, so $\cdots
\xrightarrow{} X_{n+1} \xrightarrow{d_{n+1}} X_{n} \xrightarrow{d_n}
X_{n-1} \xrightarrow{} \cdots$ is a chain complex. Given a chain complex $X \in \ch$, the
\emph{$n^{\text{th}}$ suspension of $X$}, denoted $\Sigma^n X$, is the complex given by
$(\Sigma^n X)_{k} = X_{k-n}$ and $(d_{\Sigma^n X})_{k} = (-1)^nd_{k-n}$.
For a given $R$-module $M$, we denote the \emph{$n$-disk on $M$} by $D^n(M)$. This is the complex consisting only of $M \xrightarrow{1_M} M$ concentrated in degrees $n$ and $n-1$. We denote the \emph{$n$-sphere on $M$} by $S^n(M)$, and this is the complex consisting of $M$ in degree $n$ and 0 elsewhere.

Given two chain complexes $X$ and $Y$ we define $\homcomplex(X,Y)$ to
be the complex of abelian groups $ \cdots \xrightarrow{} \prod_{k \in
\Z} \Hom(X_{k},Y_{k+n}) \xrightarrow{\delta_{n}} \prod_{k \in \Z}
\Hom(X_{k},Y_{k+n-1}) \xrightarrow{} \cdots$, where $(\delta_{n}f)_{k}
= d_{k+n}f_{k} - (-1)^n f_{k-1}d_{k}$.
This gives a functor
$\homcomplex(X,-) \mathcolon \cha{A} \xrightarrow{} \textnormal{Ch}(\Z)$. Note that this functor takes exact sequences to left exact sequences,
and it is exact if each $X_{n}$ is projective. Similarly the contravariant functor $\homcomplex(-,Y)$ sends exact sequences to left exact sequences and is exact if each $Y_{n}$ is injective. It is an exercise to check that the homology satisfies $H_n[Hom(X,Y)] = \ch(X,\Sigma^{-n} Y)/\sim$ where $\sim$ is the usual relation of chain homotopic maps.

\section{Absolutely clean chain complexes}\label{sec-abs clean}

Absolutely clean $R$-modules were defined in~\cite{bravo-gillespie-hovey}. We wish to define and characterize the analog for chain complexes of $R$-modules. First we need to characterize chain complexes of type $FP_{\infty}$.

\subsection{Chain complexes of type $FP_{\infty}$}

A module $M$ over a ring $R$ is said to be of \textbf{type
$\mathbf{FP_{\infty}}$} if $M$ has a projective resolution by finitely
generated projective modules. If $R$ is (left) Noetherian, the (left) modules of type $FP_{\infty}$ are
precisely the finitely generated modules. If $R$ is (left) coherent,
the modules of type $FP_{\infty}$ are precisely the finitely presented
modules. Bieri showed in~\cite{bieri} that for any ring $R$, the class of $FP_{\infty}$ modules is thick. This means they are closed under retracts and whenever two out of three terms in a short exact sequence $0 \rightarrow M \rightarrow N  \rightarrow  L \rightarrow 0$ are type $FP_{\infty}$ then so is the third.

\begin{definition}
A chain complex $X$ is of \textbf{type
$\mathbf{FP_{\infty}}$} if $X$ has a projective resolution $\cdots \rightarrow P_2 \rightarrow P_1  \rightarrow  P_0 \rightarrow X \rightarrow 0 $ by finitely
generated projective complexes $P_i$.
\end{definition}

Recall that by definition, a chain complex is \emph{finitely generated} if whenever $X = \Sigma_{i \in I} S_i$, for some collection $\{S_i\}_{i \in I}$ of subcomplexes of $X$, then there exists a finite subset $J \subseteq I$ for which $X = \Sigma_{i \in J} S_i$. It is a standard fact that $X$ is finitely generated if and only if it is bounded (above and below) and each $X_n$ is finitely generated. Note that a type $FP_{\infty}$ complex $X$ is certainly finitely generated since by definition it is the image of the finitely generated complex $P_0$.

\begin{proposition}\label{prop-FP-infinitiy complexes}
A chain complex $X$ is of type $FP_{\infty}$ if and only if it is bounded and each $X_n$ is an $R$-module of type $FP_{\infty}$.
\end{proposition}

\begin{proof}
Say $X$ is of type $FP_{\infty}$. Then it must be finitely generated, so it is bounded. Also, looking at the definition of a type $FP_{\infty}$ complex, we see that each complex $P_i$ must consist of finitely generated projective $R$-module in each degree. So immediately we get that each $X_n$ is of type $FP_{\infty}$ as an $R$-module.

Conversely, suppose $X$ is bounded and each $X_n$ is an $R$-module of type $FP_{\infty}$. Then it is easy to construct a surjection $f : P_0 \rightarrow X$ where $P_0$ is a finitely generated projective complex. Set $K = \ker{f}$ and note that it also must be bounded. Since each $X_n$ must also be finitely presented, it follows that each $K_n$ is finitely generated. Thus $K$ is finitely generated and we can again construct a surjection $f_1 : P_1 \rightarrow K$ where $P_1$ is a finitely generated projective complex. Set $K_1 = \ker{f_1}$ and note that $K_1$ must be bounded. Since each $X_n$ must be of type $FP_2$, it follows that $K_1$ must also be a finitely generated complex. Continuing is this way, using that $X_n$ is an $R$-module of type $FP_n$ for all $n$, we construct a projective resolution $\cdots \rightarrow P_2 \rightarrow P_1  \rightarrow  P_0 \rightarrow X \rightarrow 0 $ where each $P_i$ is a finitely
generated projective complex.
\end{proof}

\begin{corollary}\label{cor-thickness}
For any ring $R$, the class of complexes of type $FP_{\infty}$ is thick. Moreover,
\begin{enumerate}
\item $R$ is (left) Noetherian iff the finitely generated complexes coincide with the complexes of type $FP_{\infty}$.
\item $R$ is (left) coherent iff the finitely presented complexes coincide with the complexes of type $FP_{\infty}$.
\end{enumerate}

\end{corollary}

\begin{proof}
The analogous statements hold in the category $\rmod$ and so it is immediate from Proposition~\ref{prop-FP-infinitiy complexes} that they hold in $\ch$.
\end{proof}

\subsection{Absolutely clean chain complexes}

The definition below of an absolutely clean chain complex is entirely analogous to the definition of an absolutely clean $R$-module from~\cite{bravo-gillespie-hovey}.

\begin{definition}
We call a chain complex $A$ \textbf{absolutely clean} if $\Ext^1_{\ch}(X,A)=0$ for all chain complexes $X$ of type $FP_{\infty}$.
\end{definition}

Our goal now is to characterize the absolutely clean chain complexes. They will turn out to be the exact complexes $A$ for which each cycle $Z_nA$ is an absolutely clean $R$-module.

\begin{lemma}\label{lemma-abs clean lemma}
A chain complex $A$ is absolutely clean iff $\Ext^1_{\ch}(S^n(M),A)=0$ for all $R$-modules $M$ of type $FP_{\infty}$.
\end{lemma}

\begin{proof}
The proof is straightforward and is just like the argument given in~\cite[Proposition~3.17]{Ding-Chen-complex-models}. We summarize the proof: Use that any complex $X$ of type $FP_{\infty}$ must be a bounded complex of finite length $n$. Proceed by induction. The base step ($n=1$) is given, and a complex of length $n$ is an extension of a complex of length 1 by a complex of length $n-1$.
\end{proof}

\begin{proposition}\label{prop-absolutely clean chain complexes}
A chain complex $A$ is absolutely clean if and only if $A$ is exact and each $Z_nA$ is an absolutely clean $R$-module.
\end{proposition}

\begin{proof}
Say $A$ is an absolutely clean complex. Then by Lemma~\ref{lemma-abs clean lemma} we see that $\Ext^1_{\ch}(S^n(M),A) = 0$ for all $M$ of type $FP_{\infty}$. Since $R$ is of type $FP_{\infty}$, we have $\Ext^1_{\ch}(S^n(R),A)=0$ for all $n$. It follows that $A$ must be an exact complex, (for example, see~\cite[Lemma~4.5]{gillespie-degreewise-model-strucs}). Now using~\cite[Lemma~4.2]{gillespie-degreewise-model-strucs} we have $0 = \Ext^1_{\ch}(S^n(M),A) \cong \Ext^1_R(M,Z_nA)$ for all $M$ of type $FP_{\infty}$. This means each $Z_nA$ is an absolutely clean $R$-module.

On the other hand, if $A$ is exact and each $Z_nA$ is an absolutely clean $R$-module, then we can reverse this argument and apply Lemma~\ref{lemma-abs clean lemma} to conclude $A$ is absolutely clean.
\end{proof}

\subsection{Properties of absolutely clean complexes}

Absolutely clean complexes possess the same nice properties as absolutely clean $R$-modules. Here we are following~\cite[Propositions~2.5 and~2.6]{bravo-gillespie-hovey}.

A short exact sequence $0 \rightarrow X \rightarrow Y \rightarrow Z \rightarrow 0$ of chain complexes is called \textbf{pure exact} if it remains exact after applying $\Hom_{\ch}(F,-)$ for any finitely presented complex $F$. We call it \textbf{clean exact} if it has the same property but only for all $F$ of type $FP_{\infty}$ rather than all the finitely presented $F$. A subcomplex $P$ of a chain complex $X$ is called \emph{pure} (resp. \emph{clean}) if $0 \rightarrow P \rightarrow X \rightarrow X/P \rightarrow 0$ is pure exact (reps. clean exact).

\begin{proposition}\label{prop-properties of absolutely clean complexes}
For any ring $R$ the following hold:
\begin{enumerate}
\item If $A$ is an absolutely clean chain complex, then $\Ext^n_{\ch}(X,A) = 0$ for all $n > 0$ and $X$ of type $FP_{\infty}$.
\item The class of absolutely clean chain complexes is closed under pure subcomplexes and pure quotients. In fact, they are closed under clean subcomplexes and clean quotients.
\item The class of absolutely clean chain complexes is coresolving; that is, it contains the injective chain complexes and is closed under extensions and cokernels of monomorphisms.
\item The class of absolutely clean chain complexes is closed under direct products, direct sums, retracts, direct limits, and transfinite extensions.
\end{enumerate}
\end{proposition}

Recall that given a collection of chain complexes $\mathcal{D}$, we say that $X$ is a transfinite extension of objects in $\mathcal{D}$ if there is an ordinal $\lambda$ and a colimit-preserving functor $X: \lambda \to \ch$ with $X_0 \in \mathcal D$ such that each map $X_i \to X_{i+1}$ is a monomorphism whose cokernel  is in $\mathcal D$, and such that $\text{colim}_{i < \lambda} X_i  \cong X$.

\begin{proof}
Let $X$ be a chain complex of type $FP_{\infty}$ and take a resolution $P_*$ by finitely generated projective complexes. Set $X_0 = X$ and for each $i > 0$, let $X_i = \im (P_{i} \to P_{i-1})$, and thus from the following exact sequence
\[
 \cdots \to P_{i+2} \to P_{i+1} \to P_{i} \to X_i \to 0
\]
we see that each $X_i$ is of type $FP_{\infty}$. Hence by dimension shifting we get that for any absolutely clean complex $A$,
\[
 0 = \Ext^1_{\ch}(X_{n-1},A) \cong \Ext^n_{\ch}(X,A)
\]
This proves the first statement.

For the second statement, suppose that $A$ is an absolutely clean chain complex and that
\[
E \; : \; 0 \to A' \to A \to A'' \to 0
\]
is a pure exact sequence of chain complexes. So $\Hom_{\ch}(X,E)$ is exact for any finitely presented complex $X$. Then in particular $\Hom_{\ch}(X,E)$ is exact for any chain complex $X$ of type $FP_{\infty}$. Therefore $\Ext^1_{\ch}(X,A')$ is a subgroup of the zero group $\Ext^1_{\ch}(X,A)$, and thus $A'$ is an absolutely  clean chain complex. By the first part, we also have that $\Ext^2_{\ch}(X,A')=0$, and so $\Ext^1_{\ch}(X,A'')$ is the zero group too. Hence $A''$ is also an absolutely clean chain complex. Note that we only needed to assume that $E$ was a clean exact sequence for this argument to work.

Now suppose that
\[
0 \to A' \to A \to A'' \to 0
\]
is a short exact sequence with $A$ and $A'$ absolutely clean chain complexes. By applying $\Hom_{\ch}(X, -)$ to this sequence, where $X$ is of type $FP_{\infty}$, we get that $\Ext^1_{\ch}(X,A'')$ is trapped between the two zero groups $\Ext^1_{\ch}(X,A)$ and $\Ext^2_{\ch}(X,A')$, and so it is also zero. This gives us that $A''$ is an absolutely clean chain complex. Similarly, if $A'$ and $A''$ are absolutely clean chain complex, then by the same token $\Ext^1_{\ch}(X,A)=0$, whenever $X$ is of type $FP_{\infty}$. Injective complexes are easily seen to be absolutely clean, so we have proved the third statement.

Finally, for the fourth statement, observe that absolutely clean chain complexes are clearly closed under products and retracts due to standard properties of $\Ext^1$. Notice that closure under direct sums is a special case of closure under direct limits, since any direct sum is the direct limit of its finite partial sums. Also, since we already have that the absolutely clean complexes are closed under extensions, the closure under transfinite extensions will also follow from knowing closure under direct limits. Thus it is only left to show that the absolutely clean complexes are closed under direct limits. But this follows immediately from the characterization of absolutely clean complexes given in Proposition~\ref{prop-absolutely clean chain complexes} along with the corresponding fact for $R$-modules from~\cite[Proposition~2.5]{bravo-gillespie-hovey}. In other words, it follows from the fact that direct limits are exact.
\end{proof}

The next Proposition will require the following lemma whose proof can be found in~\cite[Lemma~5.2.1]{garcia-rozas} or~\cite[Lemma~4.6]{gillespie}. For a chain complex $X$, we define its cardinality to be
$|\coprod_{n \in \Z} X_n|$.

\begin{lemma}\label{lemma-pure cardinality}
Let $\kappa$ be some regular cardinal with $\kappa > |R|$.
Say $X \in \ch$ and $S \subseteq X$ has $|S| \leq \kappa$. Then there exists a pure $P \subseteq X$ with
$S \subseteq P$ and $|P| \leq \kappa$.
\end{lemma}

\begin{remark}\label{remark-homcomplex}
We note that~\cite[Lemma~5.2.1]{garcia-rozas} and~\cite[Lemma~4.6]{gillespie} give several other characterizations of pure exact sequences of complexes, but none of them are stated exactly the same as our definition above. However, they are equivalent. In particular, one of their characterizations of purity is that the altered Hom-complex functor $\overline{\homcomplex}(F,-)$ remains an exact sequence (of complexes) for any finitely presented complex $F$. However, for chain complexes $X$,$Y$, the definition of $\overline{\homcomplex}(X,Y)$ turns out to just be $\Hom_{\ch}(X,\Sigma^{-n}Y)$ in degree $n$. So indeed, $\overline{\homcomplex}(F,-)$ preserves short exact sequences if and only if $\Hom_{\ch}(F,-)$ preserves short exact sequences.
\end{remark}

\begin{proposition}\label{prop-transfinite}
Suppose $\mathcal{A}$ is a class of chain complexes that is closed under taking pure subcomplexes and quotients by pure subcomplexes. Then there is a cardinal $\kappa$ such that every chain complex in $\mathcal{A}$ is a transfinite extension of complexes in $\mathcal{A}$ with cardinality bounded by $\kappa$, meaning $\leq \kappa$.
\end{proposition}

\begin{proof}
As in Lemma~\ref{lemma-pure cardinality}, we let $\kappa$ be some regular cardinal with $\kappa > |R|$. Let $A \in \class{A}$. If $|A| \leq \kappa$ there is nothing to prove. So assume $|A| > \kappa$. We will use transfinite induction to find a strictly increasing continuous chain $A_0 \subset A_1 \subset A_2 \subset \cdots \subset A_{\alpha} \subset \cdots$
of subcomplexes of $A$ with each $A_{\alpha}, A_{\alpha+1}/A_{\alpha} \in \class{A}$ and with $|A_0| , |A_{\alpha+1}/A_{\alpha}| \leq \kappa$. We start by applying Lemma~\ref{lemma-pure cardinality} to find a pure subcomplex $A_0 \subset A$ with $|A_0| \leq \kappa$. Then $A_0$ and $A/A_0$ are each complexes in $\class{A}$ by assumption. So we again apply Lemma~\ref{lemma-pure cardinality} to $A/A_0$ to obtain a pure subcomplex $A_1/A_0 \subset A/A_0$ with $|A_1/A_0| \leq \kappa$. So far we have $A_0 \subset A_1 \subset A$ and with $A_0,A_1/A_0$ back in $\class{A}$ and with their cardinalities bounded by $\kappa$.  

We now pause to point out the important fact that $A_1 \subset A$ is also a pure subcomplex. Indeed, given a finitely presented complex $F$, we need to argue that $\Hom_{\ch}(F,A) \rightarrow \Hom_{\ch}(F,A/A_1)$ is an epimorphism. But after identifying $(A/A_0)/(A_1/A_0) \cong A/A_1$, this map is just the composite $$\Hom_{\ch}(F,A) \xrightarrow{} \Hom_{\ch}(F,A/A_1) \xrightarrow{} \Hom_{\ch}(F,(A/A_0)/(A_1/A_0)),$$ and each of these are epimorphisms because $A_0 \subset A$ is pure and $A_1/A_0 \subset A/A_0$ is pure.

Back to the increasing chain $A_0 \subset A_1 \subset A$, we also note that $A/A_1$ is back in $\class{A}$ since $(A/A_0)/(A_1/A_0) \cong A/A_1$ is a pure quotient.
So we may repeat the above procedure to construct a strictly increasing chain $A_0 \subset A_1 \subset A_2 \subset \cdots$ where each $A_n$ is a pure subcomplex of $A$ and each $A_{n+1}/A_n \in \class{A}$ has cardinality bounded by $\kappa$. We set $A_{\omega} = \cup_{n<\omega} A_n$. Then $A_{\omega}$ is also a pure subcomplex since pure subcomplexes are closed under direct unions by~\cite[pp.~3384]{gillespie}. So $A_{\omega}$ and $A/A_{\omega}$ are also each in $\class{A}$ and we may continue to building the continuous chain $$A_0 \subset A_1 \subset A_2 \subset \cdots \subset A_{\omega} \subset A_{\omega +1} \cdots$$ Continuing with transfinite induction, and setting $A_{\gamma} = \cup_{\alpha < \gamma} A_{\alpha}$ whenever $\gamma$ is a limit ordinal, this process eventually must terminate and we end up with $A$ expressed as a union of a continuous chain with all the desired properties.
\end{proof}

\begin{corollary}\label{cor-absolutely clean complexes are transfinite extensions}
There exists a cardinal $\kappa$ such that every absolutely clean chain complex is a transfinite extension of absolutely clean complexes with cardinality bounded by $\kappa$. In particular, there is a set $\class{S}$ of absolutely clean complexes for which every absolutely clean complex is a transfinite extension of ones in $\class{S}$.

\end{corollary}

\begin{proof}
Immediate from the previous  propositions. 
\end{proof}

\section{Gorenstein AC-injective chain complexes}\label{sec-Gorenstein AC-inj}

Again, following~\cite{bravo-gillespie-hovey}, we now introduce the Gorenstein AC-injective chain complexes. The main goal here is to characterize these complexes and to show that they are the fibrant objects of an injective model structure on the category $\ch$.

\begin{definition}\label{def-Gorenstein AC-injective complex}
We call a chain complex $X$ \textbf{Gorenstein AC-injective} if there exists an exact complex of injective complexes $$\cdots \rightarrow I_1 \rightarrow I_0 \rightarrow I^0 \rightarrow I^1 \rightarrow \cdots$$ with $X = \ker{(I^0 \rightarrow I^1)}$ and which remains exact after applying $\Hom_{\ch}(A,-)$ for any absolutely clean chain complex $A$.
\end{definition}

Note that it is the abelian group bifunctor $\Hom_{\ch}$ and not the \emph{complex} of abelian groups bifunctor $\homcomplex$ (see Section~\ref{sec-intro}) appearing in the above definition. However, it is equivalent to replace $\Hom_{\ch}(A,-)$ in the definition with the graded Hom-complex $\overline{\homcomplex}(A,-)$. See Remark~\ref{remark-homcomplex}. On the other hand, 
it is $\homcomplex$ that appears in the following characterization inspired by~\cite[Theorem~3.20]{Ding-Chen-complex-models}.

\begin{theorem}\label{them-characterization of Gorenstein AC-injective complexes}
A chain complex $X$ is Gorenstein AC-injective if and only if each $X_n$ is a Gorenstein AC-injective $R$-module and $\homcomplex(A,X)$ is exact for any absolutely clean chain complex $A$.
Equivalently, each $X_n$ is Gorenstein AC-injective and any chain map $f : A \rightarrow X$ is null homotopic whenever $A$ is an absolutely clean complex.
\end{theorem}

\begin{proof}
\noindent ($\Rightarrow$) Let $X$ be a Gorenstein AC-injective complex.  Then there is an exact complex of injective complexes $$\cdots \rightarrow I_1 \rightarrow I_0 \rightarrow I^0 \rightarrow I^1 \rightarrow \cdots$$ with $X = \ker{(I^0 \rightarrow I^1)}$ which remains exact after applying $\Hom_{\ch}(A,-)$ for any absolutely clean chain complex $A$. We first wish to show that each $X_n$ is a Gorenstein AC-injective $R$-module. Of course we have the exact complex of injective $R$-modules $$\cdots \rightarrow (I_1)_n \rightarrow (I_0)_n \rightarrow (I^0)_n \rightarrow (I^1)_n \rightarrow \cdots$$ and it does have $X_n = \ker{((I^0)_n \rightarrow (I^1)_n)}$. So it is left to show that this remains exact after applying $\Hom_R(A,-)$ for any absolutely clean $R$-module $A$. But given any such $A$, we get that $D^n(A)$ is absolutely clean from Proposition~\ref{prop-absolutely clean chain complexes}. Using the standard adjunction $\Hom_{\ch}(D^n(A),Y) \cong \Hom_R(A,Y_n)$, we see that the complex of abelian groups $$\cdots \rightarrow \Hom_R(A,(I_1)_n) \rightarrow \Hom_R(A,(I_0)_n) \rightarrow \Hom_R(A,(I^0)_n) \rightarrow \cdots$$ is isomorphic to the one obtained by applying $\Hom_{\ch}(D^n(A),-)$ to the original injective resolution of $X$. Since the latter complex is exact, we conclude $X_n$ is Gorenstein AC-injective.

Next we wish to show that for any absolutely clean chain complex $A$, the complex $\homcomplex(A,X)$ is exact. Since $X$ is Gorenstein AC-injective it follows from the definition that $\Ext^n_{\ch}(A,X) = 0$ for all $n\geq1$. In particular, $\Ext^1_{\ch}(A,X) = 0$ for all absolutely clean complexes $A$. This implies that the subgroup of all degreewise split extensions $\Ext^1_{dw}(A,X) = 0$ for all absolutely clean complexes $A$. So given any absolutely clean $A$, we have $H_n[\homcomplex(A,X)] \cong \Ext^1_{dw}(A,\Sigma^{-n-1}X)$ by~\cite[Lemma~2.1]{gillespie}. But it is easy to see $\Ext^1_{dw}(A,\Sigma^{-n-1}X) \cong \Ext^1_{dw}(\Sigma^{n+1}A,X)$. So for all $n$ we have $H_n[\homcomplex(A,Y)] \cong \Ext^1_{dw}(\Sigma^{n+1}A,X) = 0$ since of course $\Sigma^{n+1}A$ is also absolutely clean.

\noindent ($\Leftarrow$) Now suppose each $X_n$ is a Gorenstein AC-injective $R$-module and $\homcomplex(A,X)$ is exact for any absolutely clean chain complex $A$. We wish to construct a complete injective resolution of $X$ satisfying Definition~\ref{def-Gorenstein AC-injective complex}. We start by using that $\ch$ has enough injectives, and write a short exact sequence $0 \rightarrow X \rightarrow I^0 \rightarrow C \rightarrow 0$ where $I^0$ is an injective complex. Since the class of Gorenstein AC-injective $R$-modules is coresolving (by~\cite[Lemma~5.6]{bravo-gillespie-hovey}) we see that each $C_n$ is also Gorenstein AC-injective. We claim that $C$ also satisfies that $\homcomplex(A,C)$ is exact for all absolutely clean complexes $A$. Indeed for any choice of integers $n,k$ and an absolutely clean complex $A$, we have $\Ext^1_R(A_k,X_{k+n}) = 0$ since $A_k$ is an absolutely clean $R$-module and $X_{k+n}$ is Gorenstein AC-injective. Therefore we get a short exact sequence for all $n,k$:
$$0 \rightarrow \Hom_R(A_k,X_{k+n}) \rightarrow \Hom_R(A_k,(I^0)_{k+n}) \rightarrow \Hom_R(A_k,C_{k+n}) \rightarrow 0.$$
Since products of short exact sequences of abelian groups are again exact, we get the short exact sequence
$$0 \rightarrow \prod_{k \in \Z}\Hom_R(A_k,X_{k+n}) \rightarrow \prod_{k \in \Z}\Hom_R(A_k,(I^0)_{k+n}) \rightarrow \prod_{k \in \Z}\Hom_R(A_k,C_{k+n}) \rightarrow 0.$$
But this is degree $n$ of $0 \rightarrow \homcomplex(A,X)  \rightarrow \homcomplex(A,I^0)  \rightarrow  \homcomplex(A,C) \rightarrow 0$, and so this last sequence of complexes is short exact. Since $\homcomplex(A,X)$  and $\homcomplex(A,I^0)$ are each exact it follows that $\homcomplex(A,C)$ is also exact as claimed. Since $C$ has the same properties as $X$ we may inductively obtain an injective coresolution $$0 \xrightarrow{} X \xrightarrow{\eta} I^0 \xrightarrow{d^0} I^1 \xrightarrow{d^1} I^2 \xrightarrow{d^2} \cdots$$ where each $K^i = \ker{d^i} \ (i \geq 0)$ is degreewise Gorenstein AC-injective and which satisfies that $\homcomplex(A,K^i)$ is exact for any absolutely clean complex $A$. This coresolution must remain exact after applying $\Hom_{\ch}(A,-)$ for any absolutely clean $A$ because $\Ext^1_{\ch}(A,K^i) = \Ext^1_{dw}(A,K^i) \cong H_{-1}[\homcomplex(A,K^i)] = 0$, where again we have used~\cite[Lemma~2.1]{gillespie}.

It is left then to extend $0 \xrightarrow{} X \xrightarrow{\eta} I^0 \xrightarrow{d^0} I^1 \xrightarrow{d^1} I^2 \xrightarrow{d^2} \cdots$ to the left to obtain a complete resolution satisfying Definition~\ref{def-Gorenstein AC-injective complex}. First note that there is an obvious (degreewise split) short exact sequence $$0 \rightarrow \Sigma^{-1}X \xrightarrow{(1,-d)} \bigoplus_{n \in \Z}D^n(X_n) \xrightarrow{d+1} X \rightarrow 0.$$
Now each $X_n$ is Gorenstein AC-injective. So we certainly can find a short exact sequence
$0 \rightarrow Y_n \xrightarrow{\alpha_n} J_n \xrightarrow{\beta_n} X_n \rightarrow 0$ where $J_n$ is injective and $Y_n$ is also Gorenstein AC-injective. This gives us another short exact sequence 
\[
0 \rightarrow \bigoplus_{n \in \Z}D^n(Y_n) \xrightarrow{\bigoplus_{n \in \Z}D^n(\alpha_n)} \bigoplus_{n \in \Z}D^n(J_n) \xrightarrow{\bigoplus_{n \in \Z}D^n(\beta_n)} \bigoplus_{n \in \Z}D^n(X_n) \rightarrow 0.
\]
Notice that $\bigoplus_{n \in \Z}D^n(J_n) = \prod_{n \in \Z}D^n(J_n)$ is an injective complex and we will denote it by $I_0$. Furthermore, let $\epsilon : I_0 \rightarrow X$ be the composite 
\[
\bigoplus_{n \in \Z}D^n(J_n) \xrightarrow{\bigoplus_{n \in \Z}D^n(\beta_n)} \bigoplus_{n \in \Z}D^n(X_n) \xrightarrow{d+1} X.
\]
Then $\epsilon$ is an epimorphism since it is the composite of two epimorphisms. Moreover, setting $K_0 = \ker{\epsilon}$, it follows from the snake lemma that $K_0$ sits in the short exact sequence $0 \rightarrow \bigoplus_{n \in \Z}D^n(Y_n) \xrightarrow{} K_0 \xrightarrow{} \Sigma^{-1} X \rightarrow 0$. In particular, $K_0$ is an extension of $\bigoplus_{n \in \Z}D^n(Y_n)$ and $ \Sigma^{-1} X$, and so $K_0$ must be Gorenstein AC-injective in each degree since both of $\bigoplus_{n \in \Z}D^n(Y_n)$ and $ \Sigma^{-1} X$ are such. Because of this, if $A$ is any absolutely clean complex, applying $\homcomplex(A,-)$ to the  short exact sequence
$0 \rightarrow K_0 \xrightarrow{} I_0 \xrightarrow{} X \rightarrow 0$ will yield the short exact  $0 \rightarrow \homcomplex(A,K_0)  \rightarrow \homcomplex(A,I_0)  \rightarrow  \homcomplex(A,X) \rightarrow 0$ of chain complexes. And also $\homcomplex(A,K_0)$ must be exact since  $\homcomplex(A,I_0)$ and $\homcomplex(A,X)$ are. Since $K_0$ has the same properties as $X$, we may continue inductively to obtain the desired resolution $$\cdots \xrightarrow{} I_2 \xrightarrow{d_2} I_1 \xrightarrow{d_1} I_0 \xrightarrow{\epsilon} X \rightarrow 0$$
Finally, we paste the resolution together with $0 \xrightarrow{} X \xrightarrow{\eta} I^0 \xrightarrow{d^0} I^1 \xrightarrow{d^1} I^2 \xrightarrow{d^2} \cdots$ by setting $d_0 = \eta \epsilon$ and we are done.
\end{proof}

Following~\cite{gillespie-recoll}, in the setting of any abelian category $\cat{A}$ with enough injectives, we call a cotorsion pair $(\class{W},\class{F})$ an \textbf{injective cotorsion pair} if it is complete, $\class{W}$ is thick, and $\class{W} \cap \class{F}$ coincides with the class of injective objects. In this case, according to Hovey's correspondence from~\cite{hovey}, the cotorsion pair $(\class{W},\class{F})$ is equivalent to an abelian model structure on $\cat{A}$ where all objects are cofibrant, $\class{F}$ is the class of fibrant objects, and $\class{W}$ are the trivial objects. If $\cat{A}$ has enough projectives, we define the dual notion of a \textbf{projective cotorsion pair}.

Now let $R$ be any ring and let $\class{GI}$ denote the class of Gorenstein AC-injective $R$-modules. Set $\class{W} = \leftperp{\class{GI}}$. Then it follows from what is proved in~\cite[Section~5]{bravo-gillespie-hovey} that $(\class{W},\class{GI})$ is an injective cotorsion pair. The associated model structure generalizes the Gorenstein injective model structure defined in~\cite{hovey} and its generalization to Ding-Chen rings in~\cite{gillespie-ding}. We now show that the analog for chain complexes holds. That is, we now let $\class{GI}$ denote the class of all Gorenstein AC-injective complexes and show that these are the right half of an injective cotorsion pair in the category $\ch$.

\begin{theorem}\label{them-Gorenstein AC-injectives complete cotorsion pair}
Let $R$ be any ring and let $\class{GI}$ denote the class of Gorenstein AC-injective chain complexes. Set $\class{W} = \leftperp{\class{GI}}$.
Then $(\class{W},\class{GI})$ is an injective cotorsion pair in $\ch$. It is cogenerated by a set and so is equivalent to a cofibrantly generated (injective) model structure on $\ch$.
\end{theorem}

\begin{proof}
Again, from~\cite[Section~5]{bravo-gillespie-hovey} we know that the Gorenstein AC-injective $R$-modules are the right half of an injective cotorsion pair in $\rmod$. In particular, Proposition~5.10 of~~\cite{bravo-gillespie-hovey} shows that it is cogenerated by a some set $\class{S}_0$. That is, there is a set of $R$-modules $\class{S}_0$ such that $\rightperp{\class{S}_0}$ is the class of Gorenstein AC-injectives. We also know from Corollary~\ref{cor-absolutely clean complexes are transfinite extensions} that there exists some set $\class{S}_1$ of absolutely clean complexes such that each absolutely clean complex is a transfinite extension of ones in $\class{S}_1$. We may assume that $\class{S}_1$ is closed under suspensions. We let $\class{S} = \class{S}_1 \cup \{D^n(S) \,|\, S \in \class{S}_0 , \, n \in \Z \}$. We claim $\rightperp{\class{S}} = \class{GI}$.

($\subseteq$) Let $X \in \rightperp{\class{S}}$. Using Theorem~\ref{them-characterization of Gorenstein AC-injective complexes} we wish to show $X_n$ is a Gorenstein AC-injective $R$-module and that $\homcomplex(A,X)$ is exact whenever $A$ is an absolutely clean complex. We have that for any $S \in \class{S}_0$, $0 = \Ext^1_{\ch}(D^n(S),X)$. But $\Ext^1_{\ch}(D^n(S),X) \cong \Ext^1_R(S,X_n)$ by a standard isomorphism. Since $\class{S}_0$ cogenerates the Gorenstein AC-injective cotorsion pair, we conclude that each $X_n$ is a Gorenstein AC-injective $R$-module. Now let $A$ be an arbitrary absolutely clean complex and using Corollary~\ref{cor-absolutely clean complexes are transfinite extensions} write it as a transfinite extension $A \cong \varinjlim_{\alpha < \lambda} A_{\alpha}$ where each $A_{\alpha} \in \class{S}_1$. We have $H_n[\homcomplex(A,X)] \cong \Ext^1_{dw}(A,\Sigma^{(-n-1)}X)$ by~\cite[Lemma~2.1]{gillespie}. But $\Ext^1_{dw}(A,\Sigma^{(-n-1)}X) \cong \Ext^1_{dw}(\Sigma^{(n+1)}A,X)$ and this is a subgroup of $\Ext^1_{\ch}(\Sigma^{(n+1)}A,X)$. This latter group must be zero since each $\Ext^1_{\ch}(\Sigma^{(n+1)}A_{\alpha},X) = 0$ is zero by hypothesis. Here we are using Eklof's lemma. For a proof, see for example~\cite[Lemma~6.2]{hovey}.

($\supseteq$) Say $X$ is a Gorenstein AC-injective complex, so each $X_n$ is a Gorenstein AC-injective $R$-module and $\homcomplex(A,X)$ is exact whenever $A$ is an absolutely clean complex. So for any $S \in \class{S}_0$ we have $\Ext^1_{\ch}(D^n(S),X) \cong \Ext^1_R(S,X_n) = 0$. It is left to show $\Ext^1_{\ch}(A,X)=0$ for any $A \in \class{S}_1$. But since each such $A$ is an absolutely clean complex we have $\Ext^1_{\ch}(A,X) = \Ext^1_{dw}(A,X) \cong H_{-1}(A,X) = 0$, again using~\cite[Lemma~2.1]{gillespie}.

Having shown $\rightperp{\class{S}} = \class{GI}$, it follows that $(\leftperp{(\rightperp{\class{S}})}, \rightperp{\class{S}}) = (\class{W},\class{GI})$ is a complete cotorsion pair in $\ch$. We now show that $\class{W}$ is thick. Following the language and notation of~\cite[Definition~3.4]{gillespie-degreewise-model-strucs} it is easy to see that $(\class{W},\class{GI})$ is a \emph{degreewise orthogonal} cotorsion pair in $\ch$ and we denote by  $(\class{W}',\class{GI}')$ the corresponding cotorsion pair in $\rmod$, where $\class{GI}'$ is the class of Gorenstein AC-injective $R$-modules. It follows from~\cite[Proposition~3.7]{gillespie-degreewise-model-strucs} that $\class{W}$ consists precisely of the chain complexes $W$ with each $W_n \in \class{W}'$ and such that any chain map $f : W \rightarrow X$ is null homotopic whenever $X$ is a Gorenstein AC-injective complex. Since we already know that the Gorenstein AC-injective cotorsion pair $(\class{W}',\class{GI}')$ in $\rmod$ is injective, we already have that $\class{W}'$ is thick. It follows that $\class{W}$ is thick too. In particular, note that if $0 \rightarrow U \rightarrow V \rightarrow W \rightarrow 0$ is a short exact sequence, with any two out of three being complexes in $\class{W}$, then the third complex must also have all components in $\class{W}'$. Moreover, for any $X \in \class{GI}$, applying the functor $\homcomplex(-,X)$ yields a short exact sequence $0 \rightarrow \homcomplex(W,X) \rightarrow \homcomplex(V,X) \rightarrow \homcomplex(U,X) \rightarrow 0$ (because all $\Ext$ groups vanish degreewise). So whenever two of the three complexes here are exact, then so is the third. It follows that $\class{W}$ is thick.

It remains to show that $\class{W} \cap \class{GI}$ coincides with the class of injective chain complexes. By~\cite[Propositioin~3.6]{gillespie-recoll} it is now enough to show that $\class{W}$ contains all injective complexes. But since the Gorenstein AC-injective cotorsion pair $(\class{W}',\class{GI}')$ in $\rmod$ is injective, we already know that $\class{W}'$ contains the injective $R$-modules. Since injective complexes $I$ have the property that each $I_n$ is injective and \emph{any} chain map $I \rightarrow X$ is null homotopic, it follows from the characterization of $\class{W}$ given in the last paragraph that the injective complexes $I \in \class{W}$.
\end{proof}

\section{Level and Gorenstein AC-projective chain complexes}\label{sec-level complexes}

Having just studied absolutely clean and Gorenstein AC-injective chain complexes, we naturally wish to do the same with level and Gorenstein AC-projective chain complexes. Recall that while absolutely clean modules are defined via $FP_{\infty}$ modules and the $\Ext$ functor, the level modules are defined via $FP_{\infty}$ and the $\Tor$ functor. 
However we must be careful when generalizing to chain complexes, since the usual tensor product of chain complexes does not characterize flatness. For this we need the modified tensor product and its left derived torsion functor from~\cite{enochs-garcia-rozas} and~\cite{garcia-rozas}. So we start this section by recalling this tensor product and proving a couple of lemmas. These lemmas will then allow us to mimic our work from Sections~\ref{sec-abs clean} and~\ref{sec-Gorenstein AC-inj}.

\subsection{Modified tensor product and Tor functors}\label{subsec-modified tensor and Tor}

We denote by $X \overline{\otimes} Y$, the modified tensor product of chain complexes from~\cite{enochs-garcia-rozas} and~\cite{garcia-rozas}. This is the correct tensor product for characterizing flatness in $\ch$ since a complex $F$ is a direct limit of finitely generated projective complexes if and only if $F \overline{\otimes} -$ is an exact functor. $\overline{\otimes}$ is defined in terms of the usual tensor product $\otimes$ of chain complexes as follows. Given a complex $X$ of right $R$-modules and a complex $Y$ of left $R$-modules, we define $X \overline{\otimes} Y$ to be the complex whose $n$-th entry is $(X \otimes Y)_n / B_n(X \otimes Y)$ with boundary map  $(X \otimes Y)_n / B_n(X \otimes Y) \rightarrow (X \otimes Y)_{n-1} / B_{n-1}(X \otimes Y)$ given by 
\[
\overline{x \otimes y} \mapsto \overline{dx \otimes y}.
\]
This defines a complex and we get a bifunctor $ - \overline{\otimes} - $ which is right exact in each variable. We denote the corresponding left derived functors by $\overline{\Tor}_i$. We refer the reader to~\cite{garcia-rozas} for more details.

\begin{notation}\label{notation}
For a chain complex $X$ and an integer $n$, let $X[n]$ denote the \emph{$n$-th translation of $X$}. It is the chain complex whose degree $k$ is $X_{k-n}$, and whose differentials are unchaged. That is, $X[n]$ is the same as the $n$-th suspension, $\Sigma^n X$, but without the sign change for odd $n$. Also, for a sequence of $R$-modules $\ldots, M_{-2}, M_{-1}, M_0, M_1, M_2, \ldots$ we note that the complex $\cdots \rightarrow M_{2} \xrightarrow{0}  M_{1}  \xrightarrow{0}  M_0 \xrightarrow{0}  M_{-1} \xrightarrow{0}  M_{-2} \rightarrow  \cdots$ is isomorphic to $\bigoplus_{k\in\Z} S^k(M_k)$. We in fact will just use the notation $\bigoplus_{k\in\Z} S^k(M_k)$ to denote this complex. 
\end{notation}

\begin{lemma}\label{lemma-Tor-R}
$\overline{\Tor}_i(S^n(R),X) \cong \Sigma^{n - i} \bigoplus_{k \in \Z} S^k(H_k X)$ for all $i \geq 1$.
\end{lemma}

\begin{proof}
As explicitly shown in the diagram below, we have the projective resolution
$\cdots \rightarrow D^{n-2}(R) \rightarrow D^{n-1}(R) \rightarrow D^{n}(R) \rightarrow S^{n}(R) \rightarrow 0$ of $S^n(R)$.
\[
\begin{tikzpicture}
  \node (1) {$\cdots$};
  \node (2) [right of=1] {R};
  \node (3) [right of=2] {R};
  \node (4) [right of=3] {0};
  \node (5) [above of= 1]{$\cdots$};
  \node (6) [above of=2] {0};
  \node (7) [above of=3] {R};
  \node (8) [above of=4] {R};
  \node (9) [right of=8] {0};
  \node (10) [above of= 6]{$\cdots$};
  \node (11) [above of=7] {0};
  \node (12) [above of=8] {R};
  \node (13) [above of=9] {R};
  \node (14) [right of=13] {0};
  \node (15) [above of= 11]{$\cdots$};
  \node (16) [above of=12] {0};
  \node (17) [above of=13] {R};
  \node (18) [above of=14] {R};
  \node (19) [right of=18] {0};
  \node (20) [above of= 16]{$\cdots$};
  \node (21) [above of=17] {0};
  \node (22) [above of=18] {0};
\draw[->] (1) -- (2); \draw[-,double]  (2) -- (3); \draw[->] (3) -- (4) ;
\draw[->] (5) -- (6); \draw[->]  (6) -- (7); \draw[-, double] (7) -- (8) ; \draw[->] (8) -- (9);
\draw[->] (6) -- (2); \draw[-,double] (7) -- (3); \draw[->] (8) -- (4);
\draw[->] (11) -- (7); \draw[-,double] (12) -- (8); \draw[->] (13) -- (9);
\draw[->] (10) -- (11); \draw[->] (11) -- (12); \draw[-,double] (12) -- (13); \draw[->] (13) -- (14);
\draw[->] (16) -- (12); \draw[-,double] (17) -- (13); \draw[->] (18) -- (14);
\draw[->] (15) -- (16); \draw[->] (16) -- (17); \draw[-,double] (17) -- (18); \draw[->] (18) -- (19);
\draw[->] (20) -- (21); \draw[->] (21) -- (22); 
\draw[->] (22) -- (18); \draw[->] (21) -- (17); 
\end{tikzpicture}
\]
Denote this resolution by $\mathcal P \rightarrow S^n(R)$. To compute $\overline{\Tor}_i(S^n(R),X)$ we take the $i$-th homology of $\mathcal P \overline{\otimes} X$. Using the isomorphisms $D^n(R) \overline{\otimes} X \cong R \otimes_R X[n] \cong X[n]$ (the first isomorphism can be found in~\cite{garcia-rozas}), the complex $\mathcal P \overline{\otimes} X$ becomes
\[
\cdots \rightarrow X[n-2] \rightarrow X[n-1] \rightarrow X[n] \rightarrow 0.
\]
Looking at this resolution vertically, we see it is this:
\[
\begin{tikzpicture}
\node (1) {$X[n-3]$}; 
\node (2) [right of=1] {}; 
\node (3) [xshift=2mm,right of=2] {$\cdots$};
\node (4) [xshift=2mm,right of=3] {$X_{5-n}$};
\node (5) [xshift=4mm,right of=4] {$X_{4-n}$};
\node (6) [xshift=4mm,right of=5] {$X_{3-n}$};
\node (7) [below of=1] {$X[n-2]$}; 
\node (8) [right of=7] {}; 
\node (9) [xshift=2mm,right of=8] {$\cdots$};
\node (10) [xshift=2mm,right of=9] {$X_{4-n}$};
\node (11) [xshift=4mm,right of=10] {$X_{3-n}$};
\node (12) [xshift=4mm,right of=11] {$X_{2-n}$};
\node (13) [below of=7] {$X[n-1]$}; 
\node (14) [xshift=2mm, right of=13] {$=$}; 
\node (15) [right of=14] {$\cdots$};
\node (16) [xshift=2mm,right of=15] {$X_{3-n}$};
\node (17) [xshift=4mm,right of=16] {$X_{2-n}$};
\node (18) [xshift=4mm,right of=17] {$X_{1-n}$};
\node (19) [below of=13] {$X[n]$}; 
\node (20) [right of=19] {}; 
\node (21) [xshift=2mm,right of=20] {$\cdots$};
\node (22) [xshift=2mm,right of=21] {$X_{2-n}$};
\node (23) [xshift=4mm,right of=22] {$X_{1-n}$};
\node (24) [xshift=4mm,right of=23] {$X_{-n}$};
\node (25) [below of=19] {$0$}; 
\node (26) [right of=25] {}; 
\node (27) [right of=26] {};
\node (28) [xshift=4mm,right of=27] {$0$};
\node (29) [xshift=4mm,right of=28] {$0$};
\node (30) [xshift=4mm,right of=29] {$0$};
\node (31) [above of=1] {$\vdots$};
\node (32) [above of=4] {$\vdots$};
\node (33) [above of=5] {$\vdots$};
\node (34) [above of=6] {$\vdots$};
\node (35) [xshift=2mm,right of=6] {$\cdots$};
\node (36) [xshift=2mm,right of=12] {$\cdots$};
\node (37) [xshift=2mm,right of=18] {$\cdots$};
\node (38) [xshift=2mm,right of=24] {$\cdots$};
\node (39) [below of=28] {\scriptsize degree $2$};
\node (40) [below of=29] {\scriptsize degree $1$};
\node (41) [below of=30] {\scriptsize degree $0$};
\draw[->] (1) -- (7); \draw[->] (7) -- (13); \draw[->] (13) -- (19); \draw[->] (19) -- (25);
\draw[->] (4) -- (10); \draw[->] (5) -- (11); \draw[->] (6) -- (12); 
\draw[->] (10) -- (16); \draw[->] (11) -- (17); \draw[->] (12) -- (18); 
\draw[->] (16) -- (22); \draw[->] (17) -- (23); \draw[->] (18) -- (24); 
\draw[->] (22) -- (28); \draw[->] (23) -- (29); \draw[->] (24) -- (30);
\draw[->] (3) -- (4); \draw[->] (4) -- (5); \draw[->] (5) -- (6);  
\draw[->] (9) -- (10); \draw[->] (10) -- (11); \draw[->] (11) -- (12);  
\draw[->] (15) -- (16); \draw[->] (16) -- (17); \draw[->] (17) -- (18);  
\draw[->] (21) -- (22); \draw[->] (22) -- (23); \draw[->] (23) -- (24);  
\draw[->] (31) -- (1); \draw[->] (32) -- (4); \draw[->] (33) -- (5); \draw[->] (34) -- (6);
\draw[->] (6) -- (35); \draw[->] (12) -- (36); \draw[->] (18) -- (37); \draw[->] (24) -- (38);
\end{tikzpicture}
\]
Taking the $i^{th}$-homology (to compute $\overline{\Tor}_i(S^n(R),X)$) gives us the complex
\[
\cdots \xrightarrow{0} H_{(i+2)-n}(X) \xrightarrow{0} H_{(i+1)-n}(X) \xrightarrow{0} H_{i-n}(X) \xrightarrow{0} H_{(i-1)-n} \xrightarrow{0} \cdots
\]
where $H_{i-n}(X)$ is in degree $0$ of the complex. But this is just the chain complex $\bigoplus_{k \in \Z}S^k(H_{(i+k)-n})  = \Sigma^{n - i} \bigoplus_{k \in \Z} S^k(H_k X)$.
\end{proof}

\begin{lemma}\label{lemma-Tor-exact}
Let $X$ be an exact complex. Then for any right $R$-module $M$ and $i \geq 1$, we have
\[
\overline{\Tor}_i (S^n(M),X) \cong \bigoplus_{k \in \mathbb Z} S^k\left( \Tor_i^R (M, X_{k-n}/B_{k-n}X) \right)
\]
\end{lemma}

\begin{proof}
Take a resolution of $X$ by projective complexes $\cdots \to P'' \to P' \to P \to X \to 0$. Note that since $X$ is exact we have for each $n$, a projective resolution of $X_n/B_n X$:
\[
\cdots \to P''_n / B_n P'' \to P'_n / B_n P' \to P_n / B_n P \to X_n / B_n X \to 0.
\]
We apply $S^n(M) \overline{\otimes} -$ to $\mathcal{P} = \cdots \to P'' \to P' \to P \to 0 $ to get $\cdots \to S^n(M) \overline{\otimes} P'' \to S^n(M) \overline{\otimes} P' \to S^n(M) \overline{\otimes} P \to 0 $. We picture this complex vertically and apply the definition of $\overline{\otimes}$ to get:
\[
\begin{tikzpicture}
\node (1) {$\dfrac{M \otimes_R P'_{1-n}}{ B_{1-n}(M \otimes_R P')}$};
\node (2) [xshift=30mm,right of=1] {$\dfrac{M \otimes_R P'_{-n} }{ B_{-n}(M \otimes_R P')}$};
\node (3) [xshift=30mm,right of=2] {$\dfrac{M \otimes_R P'_{-n-1} }{ B_{-n-1}(M \otimes_R P')}$};
\node (4) [yshift=-5mm, below of=1] {$\dfrac{M \otimes_R P_{1-n} }{ B_{1-n}(M \otimes_R P)}$};
\node (5) [xshift=30mm,right of=4] {$\dfrac{M \otimes_R P_{-n} }{ B_{-n}(M \otimes_R P)}$};
\node (6) [xshift=30mm,right of=5] {$\dfrac{M \otimes_R P_{-n-1} }{ B_{-n-1}(M \otimes_R P)}$};
\node (7) [xshift=-10mm,left of=1] {$\cdots$};
\node (8) [xshift=-10mm,left of=4] {$\cdots$};
\node (9) [xshift=10mm,right of=3] {$\cdots$};
\node (10) [xshift=10mm,right of=6] {$\cdots$};
\node (11) [yshift=5mm,above of=1] {$\vdots$};
\node (12) [yshift=5mm,above of=2] {$\vdots$};
\node (13) [yshift=5mm,above of=3] {$\vdots$};
\node (14) [yshift=-3,below of=4] {\text{\tiny deg $1$}};
\node (15) [yshift=-3,below of=5] {\text{\tiny deg $0$}};
\node (16) [yshift=-3,below of=6] {\text{\tiny deg $-1$}};
\draw[->] (7)  -- (1); \draw (-1.5,0.2) node{\text{\tiny $0$}};
\draw[->] (1)  -- (2); \draw (2,0.2) node{\text{\tiny $0$}};
\draw[->] (2)  -- (3); \draw (6,0.2) node{\text{\tiny $0$}};
\draw[->] (3)  -- (9); \draw (9.5,0.2) node{\text{\tiny $0$}};
\draw[->] (8)  -- (4); \draw (-1.5,-1.3) node{\text{\tiny $0$}};
\draw[->] (4)  -- (5); \draw (2,-1.3) node{\text{\tiny $0$}};
\draw[->] (5)  -- (6); \draw (6,-1.3) node{\text{\tiny $0$}};
\draw[->] (6)  -- (10); \draw (9.5,-1.3) node{\text{\tiny $0$}};
\draw[->] (11) -- (1);
\draw[->] (12) -- (2);
\draw[->] (13) -- (3);
\draw[->] (1) -- (4);
\draw[->] (2) -- (5);
\draw[->] (3) -- (6);
\end{tikzpicture}
\]
But $P$ (and similarly for $P'$, $P''$, \ldots) is a pure exact complex and so we see that each
\[
 0 \to M \otimes_R B_nP \to M \otimes_R P_n \to M \otimes_R \left( {P_n}/{B_nP} \right) \to 0
\]
is exact, whence ${M \otimes_R P_n} / B_n(M \otimes_R P) \cong M \otimes_R \left( {P_n}/{B_nP} \right) $. Through this isomorphism, the above complex of complexes becomes
\[
\begin{tikzpicture}
\node (1) {$M \otimes_R\dfrac{ P'_{1-n}}{ B_{1-n}P'}$};
\node (2) [xshift=30mm,right of=1] {$M \otimes_R\dfrac{ P'_{-n} }{ B_{-n} P'}$};
\node (3) [xshift=30mm,right of=2] {$M \otimes_R\dfrac{ P'_{-n-1} }{ B_{-n-1} P'}$};
\node (4) [yshift=-5mm, below of=1] {$M \otimes_R\dfrac{ P_{1-n} }{ B_{1-n} P}$};
\node (5) [xshift=30mm,right of=4] {$M \otimes_R\dfrac{ P_{-n} }{ B_{-n} P}$};
\node (6) [xshift=30mm,right of=5] {$M \otimes_R\dfrac{ P_{-n-1} }{ B_{-n-1} P}$};
\node (7) [xshift=-10mm,left of=1] {$\cdots$};
\node (8) [xshift=-10mm,left of=4] {$\cdots$};
\node (9) [xshift=10mm,right of=3] {$\cdots$};
\node (10) [xshift=10mm,right of=6] {$\cdots$};
\node (11) [yshift=5mm,above of=1] {$\vdots$};
\node (12) [yshift=5mm,above of=2] {$\vdots$};
\node (13) [yshift=5mm,above of=3] {$\vdots$};
\node (14) [yshift=-3,below of=4] {\text{\tiny deg $1$}};
\node (15) [yshift=-3,below of=5] {\text{\tiny deg $0$}};
\node (16) [yshift=-3,below of=6] {\text{\tiny deg $-1$}};
\draw[->] (7)  -- (1); \draw (-1.5,0.2) node{\text{\tiny $0$}};
\draw[->] (1)  -- (2); \draw (2,0.2) node{\text{\tiny $0$}};
\draw[->] (2)  -- (3); \draw (6,0.2) node{\text{\tiny $0$}};
\draw[->] (3)  -- (9); \draw (9.5,0.2) node{\text{\tiny $0$}};
\draw[->] (8)  -- (4); \draw (-1.5,-1.3) node{\text{\tiny $0$}};
\draw[->] (4)  -- (5); \draw (2,-1.3) node{\text{\tiny $0$}};
\draw[->] (5)  -- (6); \draw (6,-1.3) node{\text{\tiny $0$}};
\draw[->] (6)  -- (10); \draw (9.5,-1.3) node{\text{\tiny $0$}};
\draw[->] (11) -- (1);
\draw[->] (12) -- (2);
\draw[->] (13) -- (3);
\draw[->] (1) -- (4);
\draw[->] (2) -- (5);
\draw[->] (3) -- (6);
\end{tikzpicture}
\]
Finally, to get $\overline{\Tor}_i (S^n(M),X)$ we take the $i^{th}$ homology of the above complex to obtain:
\[
\cdots \xrightarrow{0} \Tor_i^R \left( M , \dfrac{X_{1-n}}{B_{1-n}X} \right) \xrightarrow{0} \Tor_i^R \left( M , \dfrac{X_{-n}}{B_{-n}X} \right) \xrightarrow{0} \Tor_i^R \left( M , \dfrac{X_{-n-1}}{B_{-n-1}X} \right)  \xrightarrow{0} \cdots
\]
This is exactly $\bigoplus_{k \in \mathbb Z} S^k\left( \Tor_i^R (M, X_{k-n}/B_{k-n}X) \right)$.
\end{proof}

\subsection{Level chain complexes}

We start with the definition of a level chain complex.

\begin{definition}
We call a chain complex $L$ \textbf{level} if $\overline{\Tor}_1(X,L) = 0$ for all chain complexes $X$ of right $R$-modules of type $FP_{\infty}$.
\end{definition}

\begin{remark}
One might be bothered by the fact that our definition of absolutely clean was in terms of the vanishing of $\Ext^1_{\ch}(X,A)$, an abelian group, while our definition of level is in terms of the vanishing of  $\overline{\Tor}_1(X,L)$, a \emph{complex} of abelian groups. However, the modified tensor product $\overline{\otimes}$ of complexes makes the category of chain complexes into a closed symmetric monoidal category when one considers also the modified Hom-complex $\overline{\homcomplex}$ of~\cite{enochs-garcia-rozas} and~\cite{garcia-rozas}. The right derived functors of $\overline{\homcomplex}$, denoted $\overline{\mathit{Ext}}^i$, satisfy that $\overline{\mathit{Ext}}^i(X,Y)$ is a complex whose degree $n$ is $\Ext^i_{\ch}(X,\Sigma^{-n}Y)$. Using this it is easy to see that a complex $A$ is absolutely clean if and only if $\overline{\mathit{Ext}}^1(X,A) = 0$ for all complexes $X$ of type $FP_{\infty}$. So it is equivalent to define absolutely clean in terms of the the functor $\overline{\mathit{Ext}}^1$.
\end{remark}

We wish to characterize the level chain complexes and they will turn out to be the exact complexes $L$ for which each cycle $Z_nL$ is a level $R$-module. The proof of the following lemma is a straightforward exercise similar to Lemma~\ref{lemma-abs clean lemma}.

\begin{lemma}\label{lemma-level lemma}
A chain complex $L$ is level iff $\overline{\Tor}_1(S^n(M),L)=0$ for all right $R$-modules $M$ of type $FP_{\infty}$.
\end{lemma}

\begin{proposition}\label{prop-level chain complexes}
A chain complex $L$ is level if and only if $L$ is exact and each $Z_nL$ is a level $R$-module.
\end{proposition}

\begin{proof}
Say $L$ is a level complex. Then by Lemma~\ref{lemma-level lemma} we see that $\overline{\Tor}_1(S^n(M),L) = 0$ for all right modules $M$ of type $FP_{\infty}$. Since $R_R$ is of type $FP_{\infty}$, we have $0  = \overline{\Tor}_1(S^1(R),L)$ and so $L$ is exact from Lemma~\ref{lemma-Tor-R}.  Now by Lemma~\ref{lemma-Tor-exact} we have that
\[
0 = \overline{\Tor}_1(S^0(M),L) \cong \bigoplus_{k \in \mathbb Z} S^k\left( \Tor_1^R (M, L_{k}/B_{k}L) \right)
\]  for each right module $M$ of type $FP_{\infty}$.
Since $L$ is exact we have $L_{k}/B_{k}L \cong Z_{k-1}L$, so this shows each $Z_nL$ is a level $R$-module.

On the other hand, if $L$ is exact and each $Z_nL$ is level, then applying Lemmas~\ref{lemma-level lemma} and Lemma~\ref{lemma-Tor-exact} we conclude $L$ is level.
\end{proof}

Since $\Q$ is an injective cogenerator for the category of abelian groups, the functor $\Hom_{\Z}(-,\Q)$ preserves and reflects exactness. So Proposition~\ref{prop-level chain complexes} immediately gives us the following corollary due to the perfect duality between absolutely clean and level modules~\cite[Theorem~2.10]{bravo-gillespie-hovey}.

\begin{corollary}\label{cor-duality}
A chain complex $L$ of left (resp. right) modules is level if and only if $L^+ = \Hom_{\Z}(L,\Q)$ is an absolutely clean complex of right (resp. left) modules. And, a chain complex $A$ of left (resp. right) modules is absolutely clean if and only if $L^+ = \Hom_{\Z}(L,\Q)$ is a level complex of right (resp. left) modules.
\end{corollary}

\subsection{Properties of level complexes}

We now prove that the class of level complexes possesses very nice properties similar to those satisfied by absolutely clean complexes.

\begin{proposition}\label{prop-properties of level complexes}
For any ring $R$ the following hold:
\begin{enumerate}
\item If $L$ is a level chain complex, then $\overline{\Tor}_n(X,L) = 0$ for all $n > 0$ and chain complexes of right $R$-modules $X$ of type $FP_{\infty}$.
\item The class of level chain complexes is closed under pure subcomplexes and pure quotients. 
\item The class of level chain complexes is resolving; that is, it contains the projective chain complexes and is closed under extensions and kernels of epimorphisms.
\item The class of level chain complexes is closed under direct products, direct sums, retracts, direct limits, and transfinite extensions.
\end{enumerate}
\end{proposition}

\begin{proof}
Let $X$ be a chain complex of right $R$-modules of type $FP_{\infty}$ and take a resolution $P_*$ by finitely generated projective complexes. Set $X_0 = X$ and for each $i > 0$, let $X_i = \im (P_{i} \to P_{i-1})$, and thus from the following exact sequence
\[
 \cdots \to P_{i+2} \to P_{i+1} \to P_{i} \to X_i \to 0
\]
we see that each $X_i$ is also of type $FP_{\infty}$. Hence by dimension shifting we get that for any level complex $L$,
\[
 0 = \overline{\Tor}_1(X_{n-1},L) \cong \overline{\Tor}_n(X,L)
\]
This proves the first statement.

For the second statement, suppose that $L$ is a level chain complex and that
\[
E \; : \; 0 \to L' \to L \to L'' \to 0
\]
is a pure exact sequence of chain complexes, meaning $\Hom_{\ch}(X,E)$ is exact for any finitely presented complex $X$ (of left modules). By~\cite[Theorem~5.1.3]{garcia-rozas} this is equivalent to the statement that $X \overline{\otimes} E$ is exact for any finitely presented complex $X$ (of right $R$-modules).
So in particular $X \overline{\otimes} E$ is exact for any chain complex $X$ of right modules of type $FP_{\infty}$. Therefore $\overline{\Tor}_1(X,L'')$ must be zero since $\overline{\Tor}_1(X,L)$ is zero; thus $L''$ is a level chain complex. By the first part, we also have that $\overline{\Tor}_2(X,L'')=0$, and so $\overline{\Tor}_1(X,L')$ is the zero complex too. Hence $L'$ is also a level chain complex. (Note that it does NOT appear as though this argument will work when $E$ is just a clean exact sequence, as was the case for the class of absolutely clean complexes.)

Now for the third statement suppose that
\[
0 \to L' \to L \to L'' \to 0
\]
is a short exact sequence with $L$ and $L''$ level chain complexes. By applying $X \overline{\otimes} -$ to this sequence, where $X$ is of type $FP_{\infty}$, we get that $\overline{\Tor}_1(X,L')$ is trapped between the two zero complexes $\overline{\Tor}_1(X,L)$ and $\overline{\Tor}_2(X,L'')$, and so it is also zero. This gives us that $L'$ is a level chain complex. Similarly, if $L'$ and $L''$ are level, then we see $\overline{\Tor}_1(X,L)=0$, whenever $X$ is of type $FP_{\infty}$. Projective complexes are certainly level, so we have proved the third statement.

For the fourth statement, one can use the characterization of level complexes from Proposition~\ref{prop-level chain complexes} along with the fact that each corresponding fact is true in $R$-Mod by~\cite[Proposition~2.8]{bravo-gillespie-hovey}. For example, since $R$-Mod satisfies Grothendieck's AB4, AB4*, and AB5, exact sequences are closed under direct sums, direct products, and direct limits.  
\end{proof}

\begin{corollary}\label{cor-level complexes are transfinite extensions}
There exists a cardinal $\kappa$ such that every level chain complex is a transfinite extension of level complexes with cardinality bounded by $\kappa$. In particular, there is a set $\class{S}$ of level complexes for which every level complex is a transfinite extension of ones in $\class{S}$.
\end{corollary}

\begin{proof}
Immediate Propositions~\ref{prop-properties of level complexes} and~\ref{prop-transfinite}. 
\end{proof}

\begin{corollary}\label{cor-level complex cotorsion pair}
For any ring $R$, the class of level complexes are the left half of a complete hereditary cotorsion pair cogenerated by a set. Moreover this is a perfect cotorsion pair, meaning every complex has a level cover. 
\end{corollary}

\begin{proof}
It is a standard fact that any set $\class{S}$ in $\ch$ will cogenerate a complete cotorsion pair $(\leftperp{(\rightperp{\class{S}})},\rightperp{\class{S}})$, where $\leftperp{(\rightperp{\class{S}})}$ consists precisely of all retracts of transfinite extensions of complexes in $\class{S}$. Taking $\class{S}$ to be as in Corollary~\ref{cor-level complexes are transfinite extensions} we see that $\leftperp{(\rightperp{\class{S}})}$ is indeed the class of level complexes, because level complexes are closed under retracts and transfinite extensions by Proposition~\ref{prop-properties of level complexes}. That same proposition says that the cotorsion pair is hereditary. Since the level complexes are closed under direct limits it follows that every complex has a level cover. 
\end{proof}

\begin{remark}\label{remark-level-pair}
Let $(\class{F},\class{C})$ denote the level cotorsion pair in $R$-Mod. Then in the notation of~\cite{gillespie} the cotorsion pair of Corollary~\ref{cor-level complex cotorsion pair} is precisely $(\tilclass{F},\dgclass{C})$. This is immediate from Proposition~\ref{prop-level chain complexes}.  
\end{remark}

\subsection{Gorenstein AC-projective chain complexes}\label{sec-Gorenstein AC-proj}

Dualizing Definition~\ref{def-Gorenstein AC-injective complex}, we get the following.

\begin{definition}\label{def-Gorenstein AC-projective complex}
We call a chain complex $X$ \textbf{Gorenstein AC-projective} if there exists an exact complex of projective complexes $$\cdots \rightarrow P_1 \rightarrow P_0 \rightarrow P^0 \rightarrow P^1 \rightarrow \cdots$$ with $X = \ker{(P^0 \rightarrow P^1)}$ and which remains exact after applying $\Hom_{\ch}(-,L)$ for any level chain complex $L$.
\end{definition}

\begin{theorem}\label{them-characterization of Gorenstein AC-projective complexes}
A chain complex $X$ is Gorenstein AC-projective if and only if each $X_n$ is a Gorenstein AC-projective $R$-module and $\homcomplex(X,L)$ is exact for any level chain complex $L$.
Equivalently, each $X_n$ is Gorenstein AC-projective and any chain map $f : X \rightarrow L$ is null homotopic whenever $L$ is a level complex.
\end{theorem}

\begin{proof}
\noindent ($\Rightarrow$) Let $X$ be a Gorenstein AC-projective complex.  Then there exists an exact complex of projective complexes $$\cdots \rightarrow P_1 \rightarrow P_0 \rightarrow P^0 \rightarrow P^1 \rightarrow \cdots$$ with $X = \ker{(P^0 \rightarrow P^1)}$ and which remains exact after applying $\Hom_{\ch}(-,L)$ for any level chain complex $L$. We first wish to show that each $X_n$ is a Gorenstein AC-projective $R$-module. Of course we have the exact complex of projective $R$-modules $$\cdots \rightarrow (P_1)_n \rightarrow (P_0)_n \rightarrow (P^0)_n \rightarrow (P^1)_n \rightarrow \cdots$$ and it does have $X_n = \ker{((P^0)_n \rightarrow (P^1)_n)}$. So it is left to show that this remains exact after applying $\Hom_R(-,L)$ for any level $R$-module $L$. But given any such $L$, we get that $D^{n+1}(L)$ is level from Proposition~\ref{prop-level chain complexes}. Using the standard adjunction $\Hom_{\ch}(Y,D^{n+1}(L)) \cong \Hom_R(Y_n,L)$, we see that the complex of abelian groups $$\cdots \rightarrow \Hom_R((P^1)_n, L) \rightarrow \Hom_R((P^0)_n, L) \rightarrow \Hom_R((P_0)_n, L) \rightarrow \cdots$$ is isomorphic to the one obtained by applying $\Hom_{\ch}(-,D^{n+1}(L))$ to the original projective resolution of $X$. Since the latter complex is exact, we conclude $X_n$ is Gorenstein AC-projective.

Next we wish to show that for any level chain complex $L$, the complex $\homcomplex(X,L)$ is exact. Since $X$ is Gorenstein AC-projective it follows from the definition that, whenever $L$ is level, then $\Ext^n_{\ch}(X,L) = 0$ for all $n\geq1$. In particular, we get $\Ext^1_{dw}(X,L) = 0$ for all level complexes $L$. Since for any $n$, $\Sigma^{-n-1}L$ is also certainly level whenever $L$ is level, we get using~\cite[Lemma~2.1]{gillespie} that $0= \Ext^1_{dw}(X,\Sigma^{-n-1}L) \cong H_n[\homcomplex(X,L)]$. So $\homcomplex(X,L)$ is exact. 

\noindent ($\Leftarrow$) Now suppose each $X_n$ is a Gorenstein AC-projective $R$-module and $\homcomplex(X,L)$ is exact for any level chain complex $L$. We wish to construct a complete projective resolution of $X$ satisfying Definition~\ref{def-Gorenstein AC-projective complex}. We start by using that $\ch$ has enough projectives, and write a short exact sequence $0 \rightarrow K \rightarrow P_0 \rightarrow X \rightarrow 0$ where $P_0$ is a projective complex. Since the class of Gorenstein AC-projective $R$-modules is resolving (by~\cite[Lemma~8.6]{bravo-gillespie-hovey}) we see that each $K_n$ is also Gorenstein AC-projective. We claim that $K$ also satisfies that $\homcomplex(K,L)$ is exact for all level complexes $L$. Indeed for any choice of integers $n,k$ and a level complex $L$, we have $\Ext^1_R(X_k,L_{k+n}) = 0$ since $L_{k+n}$ is a level $R$-module and $X_{k}$ is Gorenstein AC-projective. Therefore we get a short exact sequence for all $n,k$:
$$0 \rightarrow \Hom_R(X_k,L_{k+n}) \rightarrow \Hom_R((P_0)_{k}, L_{k+n}) \rightarrow \Hom_R(K_k,L_{k+n}) \rightarrow 0.$$
Since products of short exact sequences of abelian groups are again exact, we get the short exact sequence
$$0 \rightarrow \prod_{k \in \Z}\Hom_R(X_k,L_{k+n}) \rightarrow \prod_{k \in \Z}\Hom_R((P_0)_{k}, L_{k+n}) \rightarrow \prod_{k \in \Z}\Hom_R(K_k,L_{k+n}) \rightarrow 0.$$
But this is degree $n$ of $0 \rightarrow \homcomplex(X,L)  \rightarrow \homcomplex(P_0,L)  \rightarrow  \homcomplex(K,L) \rightarrow 0$, and so this is a short exact sequence of complexes. Since $\homcomplex(X,L)$  and $\homcomplex(P_0,L)$ are each exact it follows that $\homcomplex(K,L)$ is also exact as claimed. Since $K$ has the same properties as $X$ we may inductively obtain a projective resolution $$\cdots \xrightarrow{d_3} P_2 \xrightarrow{d_2} P_1 \xrightarrow{d_1} P_0 \xrightarrow{\epsilon} X \xrightarrow{} 0$$ where each $K_i = \im{d_i} \ (i \geq 1)$ is degreewise Gorenstein AC-projective and which satisfies that $\homcomplex(K_i,L)$ is exact for any level complex $L$. This resolution must remain exact after applying $\Hom_{\ch}(-,L)$ for any level $L$ because $\Ext^1_{\ch}(K_i, L) = \Ext^1_{dw}(K_i, L) \cong H_{-1}[\homcomplex(K_i, L)] = 0$, where again we have used~\cite[Lemma~2.1]{gillespie}.

It is left then to extend $\cdots \xrightarrow{d_3} P_2 \xrightarrow{d_2} P_1 \xrightarrow{d_1} P_0 \xrightarrow{\epsilon} X \xrightarrow{} 0$ to the right to obtain a complete resolution satisfying Definition~\ref{def-Gorenstein AC-projective complex}. First note that there is an obvious (degreewise split) short exact sequence $$0 \rightarrow X \xrightarrow{(1,d)} \bigoplus_{n \in \Z}D^{n+1}(X_n) \xrightarrow{-d+1} \Sigma X \rightarrow 0.$$
Now each $X_n$ is Gorenstein AC-projective. So we certainly can find a short exact sequence
$0 \rightarrow X_n \xrightarrow{\alpha_n} Q_n \xrightarrow{\beta_n} Y_n \rightarrow 0$ where $Q_n$ is projective and $Y_n$ is also Gorenstein AC-projective. This gives us another short exact sequence 
\[
\bigoplus_{n \in \Z}D^{n+1}(X_n) \xrightarrow{ \bigoplus_{n \in \Z}D^{n+1}(\alpha_n)} \bigoplus_{n \in \Z}D^{n+1}(Q_n) \xrightarrow{\bigoplus_{n \in \Z}D^{n+1}(\beta_n)} \bigoplus_{n \in \Z}D^{n+1}(Y_n).
\]
Notice that $\bigoplus_{n \in \Z}D^{n+1}(Q_n)$ is a projective complex and we will denote it by $P^0$. Furthermore, let $\eta : X \rightarrow P^0$ be the composite 
\[
X \xrightarrow{(1,d)} \bigoplus_{n \in \Z}D^{n+1}(X_n) \xrightarrow{\bigoplus_{n \in \Z}D^{n+1}(\alpha_n)} \bigoplus_{n \in \Z}D^{n+1}(Q_n).
\]
Then $\eta$ is an monomorphism since it is the composite of two monomorphisms. Moreover, setting $C^0 = \cok{\eta}$, it follows from the snake lemma that $C^0$ sits in the short exact sequence $0 \rightarrow   \Sigma X \xrightarrow{} C^0 \xrightarrow{}  \bigoplus_{n \in \Z}D^{n+1}(Y_n)  \rightarrow 0$. In particular, $C^0$ is an extension of $\bigoplus_{n \in \Z}D^{n+1}(Y_n)$ and $\Sigma X$, and so $C^0$ must be Gorenstein AC-projective in each degree since both of $\bigoplus_{n \in \Z}D^{n+1}(Y_n)$ and $ \Sigma X$ are such. Because of this, if $L$ is any level complex, applying $\homcomplex(-,L)$ to the  short exact sequence
$0 \rightarrow X \xrightarrow{} P^0 \xrightarrow{} C^0 \rightarrow 0$ will yield the short exact  $0 \rightarrow \homcomplex(C^0,L)  \rightarrow \homcomplex(P^0,L)  \rightarrow  \homcomplex(X,L) \rightarrow 0$ of chain complexes. And also $\homcomplex(C^0,L)$ must be exact since  $\homcomplex(P^0,L)$ and $\homcomplex(X,L)$ are. Since $C^0$ has the same properties as $X$, we may continue inductively to obtain the desired resolution 
$$0 \xrightarrow{} X \xrightarrow{\eta} P^0 \xrightarrow{d^0} P^1 \xrightarrow{d^1} P^2 \xrightarrow{d^2} \cdots$$
Finally, we paste the resolution together with $\cdots \xrightarrow{} P_2 \xrightarrow{d_2} P_1 \xrightarrow{d_1} P_0 \xrightarrow{\epsilon} X \rightarrow 0$
by setting $d_0 = \eta \epsilon$ and we are done.
\end{proof}



\begin{thebibliography}{9}

\bibitem[Bie81]{bieri}
Robert Bieri, \emph{Homological dimension of discrete groups}, second ed.,
  Queen Mary College Mathematical Notes, Queen Mary College Department of Pure
  Mathematics, London, 1981. \MR{715779 (84h:20047)}

\bibitem[BGH13]{bravo-gillespie-hovey}
Daniel Bravo, James Gillespie and Mark Hovey, \emph{The stable module category of a general ring}, submitted.

\bibitem[EGR97]{enochs-garcia-rozas}
E.~Enochs and J.R. Garc\'\i{}a-Rozas, \emph{Tensor products of
chain complexes}, Math J. Okayama Univ. 39, 1997, pp. 19--42.

\bibitem[EJ01]{enochs-jenda-book}
E.~Enochs and O.~Jenda, \emph{Relative homological algebra}, De
Gruyter Expositions in Mathematics no. 30, Walter De Gruyter, New
York, 2000.

\bibitem[GR99]{garcia-rozas}
J.~R.~Garc\'\i{}a-Rozas, \emph{Covers and envelopes in the
category of complexes of modules}, Research Notes in Mathematics
no. 407, Chapman \& Hall/CRC, Boca Raton, FL, 1999.

\bibitem[Gil04]{gillespie}
James Gillespie, \emph{The flat model structure on Ch(R)}, Trans.
Amer. Math. Soc. vol.~356, no.~8, 2004, pp.~3369--3390.

\bibitem[Gil08]{gillespie-degreewise-model-strucs}
James Gillespie, \emph{Cotorsion pairs and degreewise homological model structures},
Homology, Homotopy Appl. vol.~10, no.~1, 2008, pp.~283--304.

\bibitem[Gil10]{gillespie-ding}
James Gillespie, \emph{Model structures on modules over {D}ing-{C}hen rings},
  Homology, Homotopy Appl. vol.~12, no.~1, 2010, pp.~61--73.

\bibitem[Gil13a]{gillespie-recoll}
James Gillespie, \emph{Gorenstein complexes and recollements from cotorsion pairs},  arXiv:1210.0196v2.

\bibitem[Hov02]{hovey}
 Mark Hovey, \emph{Cotorsion pairs, model category structures,
 and representation theory}, Mathematische Zeitschrift, vol.~241, 2002,
 pp.~553--592.

\bibitem[LLY13]{Ding-Chen-complex-models}
Gang Yang, Zhongkui Liu, and Li Liang, \emph{Model structures on categories of complexes over Ding-Chen rings}, Communications in Algebra, vol.~41, 2013, pp.~50--69. 


\end{thebibliography}
\end{document}